\documentclass[11pt]{article} 
\usepackage{amsmath,amssymb,amsthm,fullpage}
\usepackage{colortbl}
\usepackage{graphicx}

\theoremstyle{plain}
\newtheorem{theorem}{Theorem}
\newtheorem{lemma}{Lemma}

\theoremstyle{definition}

\newtheorem{example}{Example}

\newtheorem{definition}{Definition}
\newtheorem{remark}{Remark}

\newcommand{\al}{\alpha}

\newcommand{\de}{\delta}

\newcommand{\la}{\lambda}

\newcommand{\si}{\sigma}
\newcommand{\te}{\theta}

\newcommand{\eps}{\varepsilon}

\newcommand{\R}{\mathbb R}

\newcommand{\Sp}{{\mathbb S}^{n-1}}	

\newcommand{\brR}[1]{\!\left(#1\right)}	
\newcommand{\brS}[1]{\left[#1\right]}

\newcommand{\set}[1]{\left\lbrace#1\right\rbrace}	
\newcommand{\Set}[1]{\Big\lbrace#1\Big\rbrace}		
\newcommand{\abs}[1]{\left|#1\right|}			
		
\newcommand{\norm}[1]{\left\|#1\right\|}	
\newcommand{\rd}[1]{\rho^{}_{{#1}}}	

\newcommand{\spn}{{\rm span}}	

\newcommand{\hide}[1]{}

\begin{document}

\title{{\bf On the local convexity of intersection bodies of revolution}
}
\author{M. Angeles Alfonseca\thanks{Supported in part by U.S.~National Science Foundation grant DMS-1100657.}\,\, and\, Jaegil Kim\thanks{Supported in part by U.S.~National Science Foundation grant DMS-0652684.}}

\date{}
\maketitle

\begin{abstract} 
One of the fundamental results in Convex Geometry is Busemann's theorem, which states that the intersection body of a symmetric convex body is convex. Thus, it is only natural to ask if there is a quantitative version of Busemann's theorem, i.e., if the intersection body operation actually improves convexity. In this paper we concentrate on the symmetric bodies of revolution to provide several results on the (strict) improvement of convexity under the intersection body operation. It is shown that the intersection body of a symmetric convex body of revolution has the same asymptotic behavior near the equator as the Euclidean 
ball. We apply this result to show that  in sufficiently high dimension the double intersection body of a symmetric convex body of revolution is very close to an ellipsoid in the Banach-Mazur distance. We also prove results on the local convexity at the equator of intersection bodies in the class of star bodies of revolution.
\end{abstract}

\section{Introduction and Notation}

A set $S\subset\R^n$ is said to be {\em symmetric} if it is symmetric with respect to the origin (i.e., $S=-S$), and {\it star-shaped} if the line segment from the origin to any point in $S$ is contained in $S$. For a star-shaped set $K\subset\R^n$, the {\em radial function} of $K$  is defined by $$\rd{K}(u) = \sup \{\la \ge 0 : \la u\in K\} \quad\mbox{for every } u\in \Sp.$$  A {\em body} in $\R^n$ is a compact set which is equal to the closure of its interior.  A body $K$ is called a {\em star body} if it is star-shaped at the origin and its radial function $\rd{K}$ is continuous. We say that a body $K$ is {\em locally convex} at a point $p$ on the boundary of $K$ if there exists a neighborhood $B(p,\eps)=\set{q\in\R^n:|p-q|\le\eps}$ of $p$ such that $K\cap B(p,\eps)$ is convex. Furthermore, if $p$ is an extreme point of $K\cap B(p,\eps)$, then $K$ is said to be {\em strictly convex} at $p$.

In \cite{Lu1}, Lutwak introduced the notion of the {\em intersection body of a star body}. The intersection body $I\!K$ of a star body $K$ is defined by its radial function $$\rd{I\!K}(u)= |K\cap u^\perp| \quad \mbox{ for every } u\in \Sp.$$ Here and throughout the paper,  $u^\perp$ denotes the central hyperplane perpendicular to $u$.  By $|A|_k$, or simply $|A|$ when there is no ambiguity, we denote the $k$-dimensional Lebesgue measure of a set $A$.
From the volume formula in polar coordinates for the section $K\cap u^\perp$,  the following analytic definition of the intersection body of a star body can be derived: the radial function of the intersection body $I\!K$ of a star body $K$ is given by $$\rd{I\!K}(u) = \frac1{n-1}\int_{\Sp\cap u^\perp}\rd{K}(v)^{n-1}dv = \frac1{n-1}\brR{{\mathcal R}\rho_{K}^{n-1}}(u),\quad u\in \Sp.$$ Here ${\mathcal R}$ stands for the spherical Radon transform. The more general class of intersection bodies is defined in the following way (see \cite{Ga2}, \cite{Ko}). A star body $K$ is an {\em intersection body} if its radial function $\rd{K}$ is the spherical Radon transform of an even non-negative measure $\mu$. We refer the reader to the books \cite{Ga2}, \cite{Ko}, \cite{KoY} for more information on the definition and properties of intersection bodies, and their applications in Convex Geometry and Geometric Tomography.

In order to measure the distance between two symmetric bodies $K$ and $L$, we use the Banach-Mazur distance $$d_{B\!M}(K,L)=\inf \Set{r\ge1: K \subset TL \subset rK \text{ for some }T\in GL(n)}.$$ 
We note that the intersection bodies of linearly equivalent star bodies are linearly equivalent (see Theorem 8.1.6 in \cite{Ga2}), in the sense that $I(TK)=\abs{\det T}(T^*)^{-1}I\!K$ for any $T\in GL(n)$. This gives that $d_{B\!M}(I(TK),I(TL))=d_{B\!M}(I\!K,I\!L)$ for any $T\in GL(n)$.

A classical theorem of Busemann \cite{Bu} (see also \cite[Theorem 8.1.10]{Ga2}) states that the intersection body of a symmetric convex body is convex. In view of Busemann's theorem it is natural to ask how much of convexity is preserved or improved under the intersection body operation. As a way to measure `convexity' of a body, we may consider the Banach-Mazur distance from the Euclidean ball.  Hensley proved in \cite{He} that the Banach-Mazur distance between the intersection body of any symmetric convex body $K$ and the ball $B_2^n$ is bounded by an absolute constant $c>1$, that is, $d_{B\!M}(I\!K,B_2^n)\le c$. Compared with John's classical result, $d_{B\!M}(K,B_2^n)\le\sqrt{n}$ for any symmetric convex body $K$, we see that in many cases the intersection body operation improves convexity in the sense of the Banach-Mazur distance from the ball (see \cite{KYZ} for a similar discussion on quasi-convexity). 

Given that the intersection body of a Euclidean ball is again a Euclidean ball, another question about the intersection body operator $I$ comes from works of  Lutwak \cite{Lu2} and Gardner \cite[Prob.\ 8.6-7]{Ga2} (see also \cite{GrZ}): Are there other fixed points of the intersection body operator? It is shown in \cite{FNRZ} that in a sufficiently small neighborhood of the ball in the Banach-Mazur distance there are no other fixed points of the intersection body operator. However, in general this question is still open. 

In this paper we concentrate on the symmetric bodies of revolution to study the local convexity properties at the equator for intersection bodies. Throughout the paper we assume that the axis of revolution for any body of revolution is the $e_1$-axis. In this case $\rho^{}_K(\te)$ denotes the radial function of a body $K$ revolution at a direction whose angle from the $e_1$-axis is $\te$. Then, following \cite[Theorem C.2.9]{Ga2}, the radial function $\rd{I\!K}(\te)$ of the intersection body of $K$ is given as, for $\te\in(0,\pi\!/2]$, 
\begin{equation}\label{eq:gardner_IK}
\rd{I\!K}(\te) = \frac{c_n}{\sin\te}\int_{\pi\!/2-\te}^{\pi\!/2} \rd{K}(\varphi)^{n-1}\brS{1-\frac{\cos^2\varphi}{\sin^2\te}}^{\frac{n-4}{2}}\sin\varphi\, d\varphi
\end{equation}
and
\begin{equation}\label{eq:gardner_IK_axis}
\rd{I\!K}(0)=\abs{\rd{K}(\pi\!/2)\,B_2^{n-1}}=c_nd_n\,\sqrt{\pi\!/(2n)}\,\,\rd{K}(\pi\!/2)^{n-1},
\end{equation}
where $c_n=\frac{n-2}{n-1}\cdot\frac{2\pi^{n/2-1}}{\Gamma(n/2)}$ and $d_n=\frac{n-1}{n-2}\cdot\frac{\sqrt{n/2}\,\Gamma(n/2)}{\Gamma((n+1)/2)}\to 1$ as $n\to\infty$. Since a dilation of the body does not change its regularity or convexity, all through the paper we will replace $c_n$ by $1$.

%\blue{\\( summary on each section )}\\
In Section 2 we introduce several concepts containing the equatorial power type to describe quantitative information about convexity of bodies of revolution. %; {\em the function $\psi_K$}, {\em equatorial power type}, {\em modulus of convexity at the equator}, and so on.

In Section 3 we investigate the equatorial behavior of symmetric intersection bodies of revolution under the convexity assumption. We prove that if $K$ is a symmetric convex body of revolution, then the intersection body of $K$ has uniform equatorial power type 2, which means that its boundary near the equator is asymptotically the same as the ball. Using this result, we prove in Section 4 that if $K$ is a symmetric convex body of revolution in sufficiently high dimension, then its double intersection body is close, in the Banach-Mazur distance, to the Euclidean ball.

In Section 5 we will study the local convexity of intersection bodies at the equator without the convexity assumption. We prove that the intersection body of a symmetric star body of revolution in dimension $n\geq 5$ is locally convex at the equator, with equatorial power type 2.

\medskip
\noindent{\bf Acknowledgments}. It is with great pleasure that we thank Dmitry Ryabogin and Artem Zvavitch for helpful advice during the preparation of this paper.

\section{Equatorial power type for bodies of revolution}

The {\em equator} of a body $K$ of revolution is the boundary of the section of $K$ by the central hyperplane perpendicular to the axis of revolution, i.e., $\partial K\cap e_1^\perp$.  The goal of this section is to introduce parameters to measure the local convexity at the equator of bodies of revolution.

\begin{definition}
Let $K$ be a body of revolution in $\R^n$ about the $e_1$-axis, and let $1\le p<\infty$. Then the {\em function $\psi_K:\R\to\R^+$} is defined by 
\begin{equation}\label{eq:def.psi_K}
\psi_K(x)=\rd{K}(\te)|\sin\te| \quad\text{for }\te=\tan^{-1}(1/x).
\end{equation}
A body $K$ of revolution is said to have {\em equatorial power type $p$} if there exist constants $c_1,c_2>0$, depending on $K$, such that $c_1< |\psi_K(x)-\psi_K(0)|/x^p<c_2$ for every $x\in\R$. 
If $K$ is a symmetric convex body, then $\psi_K$ is a continuous even function which is non-increasing in $[0,+\infty)$ and $\psi_K(x)=O(1/x)$ %tends to zero 
as $x$ tends to infinity. 
\begin{figure}[h!]
\centering
\includegraphics[scale=.7,page=1]{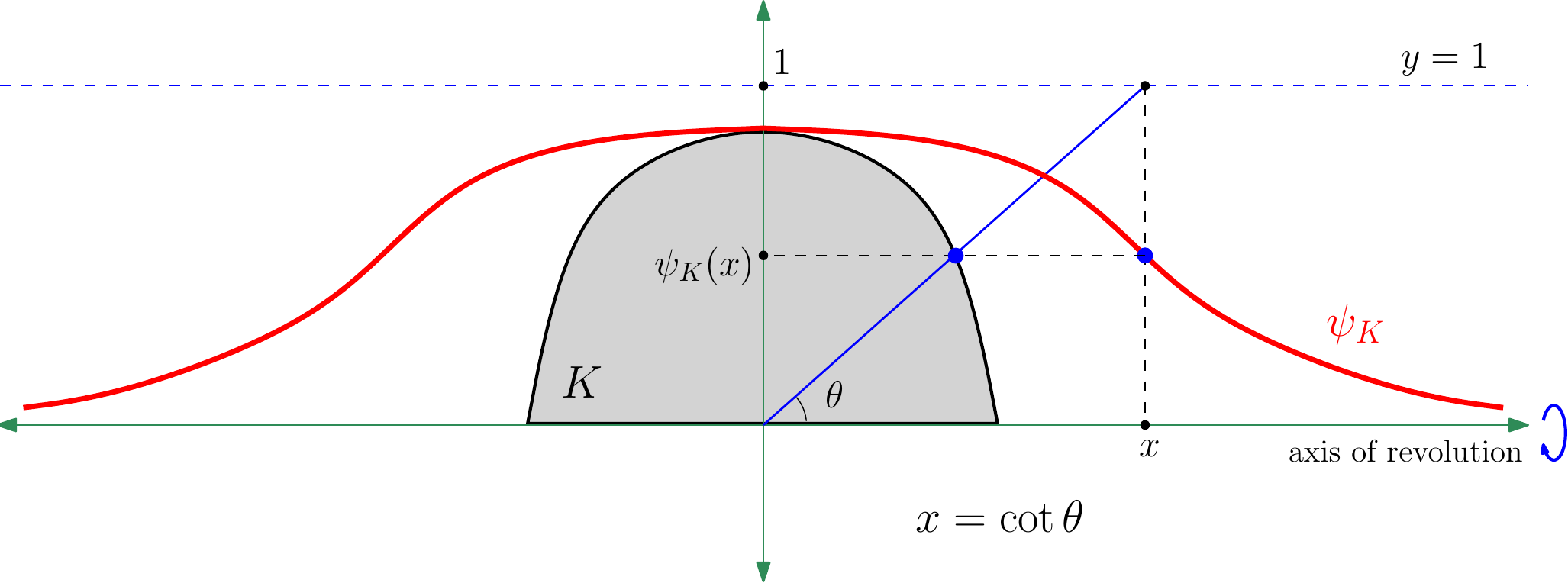}
\caption{The function $\psi_K$}
\label{psi}
\end{figure}
The local convexity properties of the function $\psi_K$ at $x=0$ are the same as those of the body $K$ at the equator.
\end{definition}

The formulas \eqref{eq:gardner_IK} and \eqref{eq:gardner_IK_axis} provide a very nice relation between $\psi_K$ and $\psi_{I\!K}$. Indeed, \eqref{eq:gardner_IK} implies
\begin{align*}
\rd{I\!K}(\te)\sin\te &=\int_0^\te \brS{\rd{K}(\pi\!/2-\phi)\cos\phi}^{n-1}\brS{\frac{1}{\cos^2\phi}}^{\frac{n-4}{2}}\brS{1-\frac{\sin^2\phi}{\sin^2\te}}^{\frac{n-4}{2}} \frac{d\phi}{\cos^2\phi} \\
&= \int_0^\te \brS{\rd{K}(\pi\!/2-\phi)\cos\phi}^{n-1}\brS{1-\frac{\tan^2\phi}{\tan^2\te}}^{\frac{n-4}{2}}\frac{d\phi}{\cos^2\phi} \\
&= \int_0^{\tan\te}\psi_K(t)^{n-1}\brS{1-t^2\cot^2\te}^{\frac{n-4}{2}}dt.
\end{align*}
Thus we have
\begin{align}
\psi_{I\!K}(0) &=\rd{I\!K}(\pi\!/2)= \int_0^\infty \psi_K(t)^{n-1}dt, \label{eq:formula2}\\
\psi_{I\!K}(x) &=\int_0^{1/x} \psi_K(t)^{n-1}\brS{1-x^2t^2}^{\frac{n-4}2}dt,\quad x\in(0,\infty). \label{eq:formula1}
\end{align}

As another way to describe equatorial power type, we may consider the classical modulus of convexity of a symmetric convex body $K$ is defined as \[\de_K(\eps)= \inf\set{1-\norm{\frac{x+y}{2}}_K: x,y\in K, \norm{x-y}_K\ge\eps},\] where $\norm{\cdot}_K$ denotes the Minkowski functional of $K$. However, since we focus on the convexity around the equator for bodies of revolution, it would be better to consider the following related notion.
\begin{definition}
Let $K$ be a symmetric convex body of revolution in $\R^n$ about the $e_1$-axis. The \emph{modulus of convexity of $K$ at the equator} is defined by
\begin{align*}
\de_K^{e}(\eps)  &= \inf\set{1-\norm{\frac{x+y}{2}}_K: x,y\in K, \norm{x-y}_K\ge2\eps, x-y\in\spn\set{e_1}}.
\end{align*}
Equivalently $\de^{e}_K$ can be expressed as
\begin{equation}\label{modconv}
\de^{e}_K(\eps)=\frac{\rd{K}(\pi\!/2)-\rd{K}(\te)\sin\te}{\rd{K}(\pi\!/2)}=\frac{\psi_K(0)-\psi_K(\cot\te)}{\psi_K(0)},
\end{equation}
where the angle $\te$ is obtained from $\eps=\frac{\rd{K}(\te)}{\rd{K}(0)}\cdot\cos\te$.
\end{definition}
It follows from \eqref{modconv} that a symmetric convex body $K$ of revolution has equatorial power type $p$ if and only if there exist constants $c_1,c_2>0$ such that $c_1<\de_K^{e}(\eps)/\eps^p<c_2$ for all $\eps\in(0,1]$. Moreover, differently from the function $\psi_K$ for a (star) body $K$ of revolution, we notice that the modulus of convexity for a convex body of revolution is invariant for any dilations on the axis of revolution or its orthogonal complement. 

For example, if $K$ is the body of revolution in $\R^n$ obtained by rotating a 2-dimensional $\ell_p$-ball with respect to the axis $e_1$, then it has equatorial power type $p$; more precisely $\de^e_K(\eps)=\eps^p/p+o(\eps^p)$.

\begin{definition}
For $1\le p<\infty$, a collection ${\mathcal C}$ of convex bodies of revolution is said to have {\em uniform equatorial power type $p$} if every convex body in ${\mathcal C}$ has equatorial power type $p$, and moreover there exist uniform constants $c_1,c_2>0$ such that 
\begin{equation*}
c_1< \frac{\de^{e}_K(\eps)}{\eps^p}<c_2\quad\text{for every $\eps\in(0,1]$ and $K\in{\mathcal C}$}.
\end{equation*}
\end{definition}
Now let us show some relation between $\de_K^e$ and $\psi_K$ when $K$ is a symmetric convex body. Fix $\eps\in(0,1]$ and choose the angle $\te\in(0,\pi\!/2)$ so that $\eps=(\rd{K}(\te)/\rd{K}(0))\cdot\cos\te$. Then, 
\begin{equation}\label{eq:convexity_K_eps}
(1-\eps)\rd{K}(\pi\!/2)\le\rd{K}(\te)\sin\te\le\rd{K}(\pi\!/2),
\end{equation}
which can be obtained by applying convexity property of $K$ to three points on the boundary of $K$ with angles $0$, $\te$, and $\pi\!/2$. Notice that \eqref{eq:convexity_K_eps} gives $$\cot\te=\frac{\rd{K}(0)\cdot\eps}{\rd{K}(\te)\sin\te}={\rd{K}(0)\over\rd{K}(\pi\!/2)}\brR{\eps + O(\eps^2)}.$$
Next it follows from \eqref{modconv} that 
\begin{equation}\label{eq:delta.psi_formula}
\de^e_K(\eps)={\psi_K(0)-\psi_K(\de)\over\psi_K(0)},\quad\text{for }\de={\rd{K}(0)\over\rd{K}(\pi\!/2)}\brR{\eps +O(\eps^2)}.
\end{equation}
In particular, we have $\de^e_K(\eps)\approx1-\psi_K(\eps)$ under the assumption $\rd{K}(0)=\rd{K}(\pi\!/2)=1$. 

In Section~\ref{sec:powertype2} we prove that the class of all intersection bodies of symmetric convex bodies of revolution have uniform equatorial power type 2, and we also provide an example showing that the convexity condition cannot be dropped. Thus, for star bodies of revolution, it is not necessary to consider $\de^e_K$ as an invariant quantity under dilations on the axis $e_1$ or its orthogonal complement; the function $\psi_K$ will be enough for star bodies. For symmetric convex bodies of revolution, we will use the modulus $\de^e_K(\eps)$ of convexity at the equator to describe the power type or the asymptotic behavior at the equator, and moreover the function $\psi_K(x)$ may be used to compute $\de^e_K(\eps)$ by \eqref{eq:delta.psi_formula}.

\section{Uniform equatorial power type 2 for intersection bodies}\label{sec:powertype2}

In this section we prove that the class of all intersection bodies of symmetric convex bodies of revolution has uniform equatorial power type 2. Namely, if $K$ is a symmetric convex body of revolution, then $I\!K$ has equatorial power type $2$ and, moreover, the coefficient of the quadratic term in the expansion of $\de^e_{I\!K}(\eps)$ is bounded above and below by absolute constants. 

First we need a specific formula for the function $\psi_K$ in the case that $K$ is a symmetric body of revolution obtained by rotating line segments.

\begin{lemma}\label{lem:easy_bounds_psi_K}
Let $L_{a,b}\subset\R^n$ be the symmetric body of revolution whose boundary is determined by a line segment $\set{(x,y):ax+by=1, 0\le x\le 1/a}$ for $a,b\ge0$. Namely, the body $L_{a,b}$ can be given by $$L_{a,b}=\set{(x,y)\in\R^n=\R\times\R^{n-1}:a|x|+b|y|\le1}.$$ Then the function $\psi_{L_{a,b}}$, defined in \eqref{eq:def.psi_K}, is equal to
\begin{equation}\label{eq:psi_for_line}
\psi_{L_{a,b}}(x)=\frac1{a|x|+b}.
\end{equation}
Moreover, if $K\subset\R^n$ is a symmetric convex body of revolution with $\rd{K}(0)=\rd{K}(\pi\!/2)=1$, then
\begin{equation}\label{eq:easy_bounds_psi_K}
\frac1{|x|+1}\le\psi_K(x)\le\min\brR{1,\frac1{|x|}}.
\end{equation}
\end{lemma}
\begin{proof}
Let $x=\cot\te >0$, and write $L=L_{a,b}$ to shorten the notation. Then the point $$\brR{\rd{L}(\te)\cos\te, \rd{L}(\te)\sin\te}\in\R^2$$ lies on the straight line $\set{(x,y)\in\R^2:ax+by=1}$. Thus we have 
\begin{align*}
x &= \cot\te=\frac{\rd{L}(\te)\cos\te}{\rd{L}(\te)\sin\te}=\frac{(1-b\rd{L}(\te)\sin\te)/a}{\rd{L}(\te)\sin\te}\\
&=\frac{1-b\,\psi_L(x)}{a\,\psi_L(x)},
\end{align*}
which gives $\psi_L(x)=1/(ax+b)$.

Now, if $K\subset\R^n$ is a symmetric convex body of revolution with $\rd{K}(0)=\rd{K}(\pi\!/2)=1$, then we have $B_1\subset K\subset B_\infty$ where
\begin{align*}
B_1\, &=\set{(x,y)\in\R^n=\R\times\R^{n-1}:|x|+|y|\le1}\tag{double cone}\\
B_\infty&=\set{(x,y)\in\R^n=\R\times\R^{n-1}:|x|\le1, |y|\le1}.\tag{cylinder}
\end{align*}
We also see that $\psi_{B_1}\le\psi_K\le\psi_{B_\infty}$ by definition of the function $\psi$.  Here $\psi_{B_1}$ and $\psi_{B_\infty}$ can be obtained from \eqref{eq:psi_for_line}: $$\psi_{B_1}(x)=\psi_{L_{1,1}}(x)=\frac1{x+1}$$ and 
$$
\psi_{B_\infty}(x)=
\begin{cases}
\psi_{L_{0,1}}(x)=1, &\text{if } 0\le x\le1 \\
\psi_{L_{1,0}}(x)=1/x, &\text{if } x\ge1
\end{cases}
$$
which imply \eqref{eq:easy_bounds_psi_K}.

\end{proof}

The inequality \eqref{eq:easy_bounds_psi_K} in Lemma \ref{lem:easy_bounds_psi_K} gives an easy upper or lower bound for the function $\psi_K$. However, we will need better other bounds for high dimension, which are given by the following Lemma. 

\begin{lemma}\label{lem:convexity}
Let $K\subset\R^n$ be a convex body of revolution with $\rd{K}(\pi\!/2)=1$.
For every $\si>0$ and $t>1$, 
\begin{equation}\label{eq:upperbound_of_psiK}
\psi_K(\si t)\le \brS{1+t\brR{\frac{1}{\psi_K(\si)}-1}}^{-1}.
\end{equation}
\end{lemma}
\begin{proof}
Let $\phi_1=\tan^{-1}(1/\si)$ and $\phi_2=\tan^{-1}(1/\si t)$. Choose three points $P_0,P_1,P_2\in\partial K\cap\spn\set{e_1,e_2}$ whose angles from the $e_1$-axis are $\pi\!/2$, $\phi_1$ and $\phi_2$, respectively. That is, 
\begin{eqnarray*}
P_0 &=& \brR{0,1} \\
P_1 &=& \brR{\rd{K}(\phi_1)\cos\phi_1,\rd{K}(\phi_1)\sin\phi_1}=:(x_1,y_1) \\
P_2 &=& \brR{\rd{K}(\phi_2)\cos\phi_2,\rd{K}(\phi_2)\sin\phi_2}=:(x_2,y_2). 
\end{eqnarray*}
Since $\frac{1-y_2}{1-y_1}\ge\frac{x_2}{x_1}$ by convexity of $K$,
\begin{equation*}
\frac{1-\rd{K}(\phi_2)\sin\phi_2}{1-\rd{K}(\phi_1)\sin\phi_1} \ge \frac{\rd{K}(\phi_2)\cos\phi_2}{\rd{K}(\phi_1)\cos\phi_1}=\frac{\rd{K}(\phi_2)\sin\phi_2\cdot \cot\phi_2}{\rd{K}(\phi_1)\sin\phi_1\cdot \cot\phi_1},
\end{equation*}
which implies $$\frac{1-\psi_K(\si t)}{1-\psi_K(\si)}\ge \frac{\psi_K(\si t)\si t}{\psi_K(\si)\si}.$$ Simplifying the above inequality, we have the inequality \eqref{eq:upperbound_of_psiK}.
\end{proof}

The next Lemma will be helpful to bound the integral in \eqref{eq:formula2}, and to control its tail.
\begin{lemma}\label{lem:approx}
Let $K\subset\R^n$, for $n\ge4$, be a convex body of revolution with $\rd{K}(\pi\!/2)=1$, and let $\si_K=\psi_K^{-1}(1-1/n)$. Then 
\begin{equation}\label{eq:estimate_integral_psi}
c_1\le \frac1{\si_K}\int_0^\infty\psi_K(t)^{n-1} dt\le c_2,
\end{equation}
where $c_1$, $c_2>0$ are absolute constants.
In addition, for every $R>1$,
$$\int_R^\infty \psi_K(\si_Kt)^{n-1} dt = O\brR{\brS{1+{R}/{n}}^{2-n}}.$$ Here, $f(\eps)=O(\eps)$  means that $|f(\eps)|\le c\eps$ for small $\eps>0$ and an absolute constant $c>0$.
\end{lemma}
\begin{proof}
For any $t\ge R$, Lemma \ref{lem:convexity} gives 
\begin{equation*}
\psi_K(\si_K t)\le \brS{1+\frac{t}{\psi_K(\si_K)}-t}^{-1}= \brS{1+\frac{t}{n-1}}^{-1}.
\end{equation*}
Thus
\begin{align*}
\int_R^\infty \psi_K(\si_Kt)^{n-1} dt &\le \int_R^\infty \brS{1+\frac{t}{n-1}}^{1-n} dt = \frac{n-1}{n-2}\brS{1+\frac{R}{n-1}}^{2-n}.
\end{align*}
Next we will show an upper bound in \eqref{eq:estimate_integral_psi}, 
\begin{align*}
\int_0^\infty\psi_K(t)^{n-1}dt &=\int_0^{\si_K}\psi_K(t)^{n-1} dt + \int_{\si_K}^\infty \psi_K(t)^{n-1}dt\\
&\le \si_K + \si_K\int_1^\infty \psi_K(\si_Kt)^{n-1} dt \\
&= \si_K + \si_K \frac{n-1}{n-2}\brS{1+\frac{1}{n-1}}^{2-n} \to (1+1/e)\si_K \quad\text{as }n\to\infty.
\end{align*}
For a lower bound,
$$\int_0^{\si_K} \psi_K(t)^{n-1} dt \ge \int_0^{\si_K} \brS{1-\frac1n}^{n-1} dt \quad\to \frac{\si_K}{e}\quad\text{as }n\to\infty.$$
Thus we have that $c_1\si_K\le\int_0^\infty\psi_K(t)^{n-1} dt\le c_2\si_K$ for absolute constants $c_1$, $c_2>0$.
\end{proof}

Next, Lemma \ref{lem:formula_R} will allow us to estimate the integral in \eqref{eq:formula1}.
\begin{lemma}\label{lem:formula_R}
Let $K\subset\R^n$, for $n\ge4$, be a symmetric convex body of revolution with $\rd{K}(\pi\!/2)=1$. Fix $R>1$ and let $\si_K=\psi_K^{-1}(1-1/n)$. Then, for each $x\le\frac1{R\si_K}$,
$$\frac{\psi_{I\!K}(x)}{\si_K} = \int_0^R \psi_K(\si_Kt)^{n-1}\brS{1-\si_K^2x^2t^2}^{\frac{n-4}2}dt  + O\brR{\brS{1+R/n}^{2-n}}.$$
\end{lemma}
\begin{proof}
Apply  Lemma \ref{lem:approx}.
If $x\neq0$, then
\begin{align*}
\psi_{I\!K}(x) &= \int_0^{1/x} \psi_K(t)^{n-1}\brS{1-x^2t^2}^{\frac{n-4}2}dt \\
&\le\si_K\int_0^R \psi_K(\si_Kt)^{n-1}\brS{1-\si_K^2x^2t^2}^{\frac{n-4}{2}}dt  + \si_K\int_R^\infty \psi_K(\si_Kt)^{n-1}dt\\
&=\si_K\brS{\int_0^R \psi_K(\si_Kt)^{n-1}\brS{1-\si_K^2x^2t^2}^{\frac{n-4}{2}}dt  + O\brR{\brS{1+R/n}^{2-n}}}. 
\end{align*}
If $x=0$, then
\begin{align*}
\psi_{I\!K}(0) &= \int_0^\infty \psi_K(t)^{n-1}dt =\si_K\brS{\int_0^R \psi_K(\si_Kt)^{n-1}dt  + O\brR{\brS{1+R/n}^{2-n}}}. 
\end{align*}
\end{proof}

Now we are ready to prove the main result of this section.

\begin{theorem}\label{thm:uniform_convexity_IK}
The class of intersection bodies of symmetric convex bodies of revolution in dimension $n \geq 4$ has uniform equatorial power type 2.  Namely, if $K$ is a symmetric convex body of revolution in $\R^n$ for $n\ge4$, then its intersection body $I\!K$ has modulus of convexity at the equator of the form $$\de^e_{I\!K}(\eps)=c_K\eps^2+O(\eps^3)$$ where $c_K>0$ is a constant depending on $K$ bounded above and below by absolute constants.
\end{theorem}

\begin{proof}
The modulus of convexity at the equator is invariant for any dilations on the axis of revolution or its orthogonal complement, so we may start with $\rd{K}(\pi\!/2)=\rd{K}(0)=1$. Fix a small number $\eps>0$ and choose the angle $\te$ such that $${\rd{I\!K}(\te)\over\rd{I\!K}(0)}\,\cos\te=\eps.$$
Let $\de=\cot\te$ and $\si_K=\psi_K^{-1}(1-1/n)$. By Lemma \ref{lem:approx}, we have 
\begin{equation}\label{eq:psiIK.zero}
\psi_{I\!K}(0)=\int_0^\infty \psi_K(t)^{n-1}dt = d_K\si_K,
\end{equation}
where $c_1\le d_K\le c_2$ for absolute constants $c_1,c_2>0$.
By convexity of $I\!K$, as in \eqref{eq:convexity_K_eps},
\begin{equation}\label{eq:convexity_IK_pf_thm}
(1-\eps)\rd{I\!K}(\pi\!/2) \le \rd{I\!K}(\te)\sin\te \le \rd{I\!K}(\pi\!/2).
\end{equation}
Note that $\rd{I\!K}(\pi\!/2)=\psi_{I\!K}(0)= d_K\si_K$ and $\rd{I\!K}(\te)\sin\te=\brR{{\rd{I\!K}(\te)\over\rd{I\!K}(0)}\cos\te}\rd{I\!K}(0)\tan\te=(\eps/\de)\rd{I\!K}(0)$. Since $\rd{I\!K}(0)/\sqrt{\pi/(2n)}$ tends to 1 as $n\to\infty$  by \eqref{eq:gardner_IK_axis}, the inequality \eqref{eq:convexity_IK_pf_thm} implies that there exists absolute constants $c_1, c_2>0$ such that
\begin{equation}\label{eq:theta}
c_1\si_K\sqrt{n} \le\eps/\de\le c_2\si_K\sqrt{n}.
\end{equation}
First consider the case of $n\ge14$. The formula (\ref{eq:theta}) and Lemma \ref{lem:formula_R} give that, for any $R$ with $1\le R\le(\si_K\de)^{-1}$,
\begin{align}\label{eq:size.of.error}
\frac{\psi_{I\!K}(\de)}{\si_K} &= \int_0^R \psi_K(\si_Kt)^{n-1}\brS{1-(\si_K\de t)^2}^{\frac{n-4}2}dt  + O\brR{\brS{1+R/n}^{2-n}}\notag\\
&= \int_0^R \psi_K(\si_Kt)^{n-1}\brS{1-\frac{n-4}2(\si_K\de t)^2+O\brR{\abs{n(\si_K\de R)^2}^2}}dt  + O\brR{\brS{1+R/n}^{2-n}}
\end{align}
and \[\frac{\psi_{I\!K}(0)}{\si_K} = \int_0^R \psi_K(\si_Kt)^{n-1}dt  + O\brR{\brS{1+R/n}^{2-n}}.\]
Since $\si_K\de$ is comparable to $\eps/\sqrt{n}$ by absolute constants by \eqref{eq:theta}, we can take $R=\eps^{-1/4}$ to control error terms of above equation. Then we have $\abs{n(\si_K\de R)^2}^2=O(\eps^4R^4)=O(\eps^3$), and for $n\ge14$,
\begin{align*}
(1+R/n)^{2-n}&\le(1+R/n)^{-12}=\brS{\frac{n}{1+n\eps^{1/4}}}^{12}\eps^3\\
(1+R/n)^{2-n}&\to e^{-R}=e^{-\eps^{-1/4}}\quad\text{as }n\to\infty,
\end{align*}
so the remainder part of \eqref{eq:size.of.error} is $O(\eps^3)$ for $n\ge14$. Thus,
\begin{align}\label{eq:psiIK.delta}
\frac{\psi_{I\!K}(\de)}{\si_K} &= \int_0^R\psi_K(\si_Kt)^{n-1}dt - \frac12(n-4)(\si_K\de)^2\int_0^R\psi_K(\si_Kt)^{n-1}t^2dt+O(\eps^3) \notag\\
&=\frac{\psi_{I\!K}(0)}{\si_K} - \frac{\pi}{4d_K^2}\eps^2\int_0^R\psi_K(\si_Kt)^{n-1}t^2dt+O(\eps^3).
\end{align}
The formula \eqref{eq:delta.psi_formula}, together with \eqref{eq:psiIK.zero} and \eqref{eq:psiIK.delta}, gives the modulus of convexity at the equator as follows.
\begin{align*}
\de^e_{I\!K}(\eps) &= {\psi_{I\!K}(0)-\psi_{I\!K}(\de)\over\psi_{I\!K}(0)}=\frac{\pi}{4d_K^3} \eps^2\int_0^{R} \psi_K(\si_Kt)^{n-1}t^2dt + O(\eps^3).
\end{align*}
Now it is enough to compute $\int_0^{R} \psi_K(\si_Kt)^{n-1}t^2dt$. 
For a upper bound, apply Lemma \ref{lem:convexity}. Then, for any $t\ge1$, 
\begin{equation*}
\psi_K(\si_K t)\le \brS{1+\frac{t}{\psi_K(\si_K)}-t}^{-1}= \brS{1+\frac{t}{n-1}}^{-1}.
\end{equation*}
Thus
\begin{align*}
\int_0^{R}\psi_K(\si_Kt)^{n-1}t^2dx &\le \int_0^1dt + \int_1^\infty t^2\brS{1+\frac{t}{n-1}}^{1-n}\!\!dt = 1+ (n-1)^3\int_{\frac{n}{n-1}}^\infty (s-1)^2s^{1-n}ds\\
&=1+\brR{1-\frac1n}^{n-1}\!\!\!\frac{5n^3-15n^2+12n}{(n-2)(n-3)(n-4)} \to 1+\frac5e \quad\text{as }n\to\infty.
\end{align*}
For a lower bound,
\begin{equation*}
\int_0^{R} \psi_K(\si_Kt)^{n-1}t^2dt \ge\int_0^1 \psi_K(\si_Kt)^{n-1}t^2dt \ge\int_0^1 \Big[1-\frac1n\Big]^{n-1}t^2dt \to \frac1{3e} \quad\text{as }n\to\infty,
\end{equation*}
which completes the proof for $n\ge14$.\\

Now consider the case of $4\le n<14$. It follows from \eqref{eq:formula2} and \eqref{eq:formula1} that
\begin{align*}
\psi_{I\!K}(0)-\psi_{I\!K}(\de) &=\int_0^\infty \psi_K(t)^{n-1}dt - \int_0^\infty \psi_K(t)^{n-1}(1-\de^2t^2)^{\frac{n-4}2}dt \\
&=\int_{1/\de}^\infty \psi_K(t)^{n-1}dt + \int_0^{1/\de} \psi_K(t)^{n-1}\brS{1-(1-\de^2t^2)^{\frac{n-4}2}}dt \\
&= \text{\quad(I)\qquad\quad+\qquad\quad(II)}. 
\end{align*}
For (I), use the inequalities \eqref{eq:easy_bounds_psi_K} from Lemma \ref{lem:easy_bounds_psi_K}. Then $$\text{(I)}\le\int_{1/\de}^\infty\frac1{t^{n-1}}dt=\frac{\de^{n-2}}{n-2}=O(\de^{n-2})$$ and $$\text{(I)}\ge\int_{1/\de}^\infty\frac1{(t+1)^{n-1}}dt=\frac{(1+1/\de)^{2-n}}{n-2}=O(\de^{n-2}).$$ Since $\de$ is comparable to $\eps$ by \eqref{eq:theta}, we have that (I) is $O(\eps^{n-2})$, which is at most $O(\eps^3)$ if $n\ge5$.

To get an upper bound of (II), use \eqref{eq:easy_bounds_psi_K} again.
\begin{align*}
\text{(II)} &\le \int_0^1 \brS{1-(1-\de^2t^2)^{\frac{n-4}2}}dt + \int_1^{1/\de} (1/t)^{n-1}\brS{1-(1-\de^2t^2)^{\frac{n-4}2}}dt \\
&= \text{\qquad\quad(II-1)\qquad\quad+\qquad\quad(II-2)},
\end{align*}
where $$\text{(II-1)}=\int_0^1 \brS{1-(1-\frac{n-4}2\de^2t^2)}dt + O(\de^4)=\frac{n-4}6\,\de^2 + O(\de^4)$$ and
\begin{align*}
\text{(II-2)} &=\int_1^{1/\de} (1/t)^{n-1}dt - \int_1^{1/\de} \frac{(1-\de^2t^2)^{\frac{n-4}2}}{t^{n-4}}\,\frac{dt}{t^3}\\
&= \frac{1-\de^{n-2}}{n-2}-\frac12\int_{\de^2}^1(s-\de^2)^{\frac{n-4}2}ds = \frac{1-\de^{n-2}-(1-\de^2)^{\frac{n-2}2}}{n-2}\\
&=\frac12\de^2+O(\de^{n-2}).
\end{align*}
A lower bound of (II) is given by
\begin{align*}
\text{(II)} &\ge \int_0^{1/\de}\frac{1-(1-\de^2t^2)^{\frac{n-4}2}}{(t+1)^{n-1}}dt \\
&= \int_0^1\frac{1-(1-\de^2t^2)^{\frac{n-4}2}}{(t+1)^{n-1}}dt + \int_0^{1/\de}\brR{\frac{t}{t+1}}^{n-1}\frac{1-(1-\de^2t^2)^{\frac{n-4}2}}{t^{n-1}}dt \\
&\ge \frac1{2^{n-1}}\int_0^1\brS{1-(1-\de^2t^2)^{\frac{n-4}2}}dt + \frac1{2^{n-1}}\int_0^{1/\de}\frac{1-(1-\de^2t^2)^{\frac{n-4}2}}{t^{n-1}}dt \\
&= \frac{n-4}{3\cdot 2^n}\de^2 +\frac1{2^{n-1}}\text{(II-2)} = \brR{\frac{n-4+3\cdot2^{n-1}}{3\cdot 2^n}}\de^2 +O(\de^{n-2}).
\end{align*}
In summary, if $n\ge5$, then (I) is at most $O(\eps^3)$ and (II) is asymptotically $c\de^2+O(\de^3)$. In addition, if $n=4$, then (II) disappears and (I) is $c\de^2+O(\de^3)$. Note that $\de^e_{I\!K}(\eps) = \frac{\psi_{I\!K}(0)-\psi_{I\!K}(\de)}{\psi_{I\!K}(0)}$ and $c_1<\psi_{I\!K}(0)<c_2$ by Lemma \ref{lem:approx}. Finally, we get $$c_1'<\de^e_{I\!K}(\eps)/\eps^2<c_2',$$ where $c_1'$, $c_2'$ are positive absolute constants.
\end{proof}

\begin{remark}
In general, Theorem \ref{thm:uniform_convexity_IK} is not true in dimension $3$. For example, the intersection body of the double cone $B_1\subset\R^3$ does not have equatorial power type $2$.
\end{remark}
\begin{proof}
It follows from Lemma \ref{lem:easy_bounds_psi_K} that the function $\psi_K$ for the double cone $K=B_1$ is given by $\psi_{B_1}(x)=\frac1{x+1}$. Let $\eps=\brR{\rd{I\!B_1}(\te)/\rd{I\!B_1}(0)}\cos\te$ for some angle $\te$ and let $\de=\cot\te$. Then
\begin{align*}
\psi_{I\!B_1}(0)&=\int_0^\infty \psi_{B_1}(t)^2dt=\int_0^\infty (t+1)^{-2}dt=1\\
\psi_{I\!B_1}(\de) &=\int_0^{1/\de} \psi_{B_1}(t)^2(1-\de^2t^2)^{-1/2}dt=\int_0^{1/\de}\frac{dt}{(t+1)^2\sqrt{1-\de^2t^2}}.
\end{align*}
So, 
\begin{align*}
\de^e_{I\!B_1}(\eps)&= 1-\psi_{I\!B_1}(\de)=\int_0^\infty \frac1{(t+1)^2}dt-\int_0^{1/\de}\frac{dt}{(t+1)^2\sqrt{1-\de^2t^2}}\\
&=\int_{1/\de}^\infty \frac1{(t+1)^2}dt+\int_0^{1/\de}\frac1{(t+1)^2}\brR{1-\frac1{\sqrt{1-\de^2t^2}}}dt\\
&=\brR{\de-\frac{\de^2}{1+\de}}-\de\int_0^1\frac{t^2dt}{(t+\de)^2\sqrt{1-t^2}(1+\sqrt{1-t^2})} \\
&=\de -\de f(\de) + O(\de^2),
\end{align*}
where $$f(\de)=\int_0^1\frac{t^2dt}{(t+\de)^2\sqrt{1-t^2}(1+\sqrt{1-t^2})}.$$ Note that 
\[f(0)=\int_0^1\frac{dt}{\sqrt{1-t^2}(1+\sqrt{1-t^2})}=1\] and \[\lim_{\de\rightarrow0}\frac{f(0)-f(\de)}{\de}=\int_0^1\frac{dt}{t\sqrt{1-t^2}(1+\sqrt{1-t^2})}=\infty,\]
Since $\de$ is comparable to $\eps$, we conclude that $\de^e_{I\!B_1}(\eps)=o(\eps)$, but $\de^e_{I\!B_1}(\eps)\neq O(\eps^2)$.
\end{proof}

\begin{remark}
The convexity condition of $K$ in Theorem \ref{thm:uniform_convexity_IK} is crucial to get the uniform boundedness of the constant $c_K$. %Namely, for some star body of revolution, we cannot get the uniform boundedness of the constant $c_K$. 
For $t>0$, consider the star body of revolution $K_t$, defined as the union $K_t=L_t\cup B_\infty$ of two cylinders $$L_t=\set{(x,y)\in\R^n=\R\times\R^{n-1}:|x|\le e^{-1/t}, |y|\le1/t}$$ and $$B_\infty=\set{(x,y)\in\R^n=\R\times\R^{n-1}:|x|\le1, |y|\le1}.$$  
If $t>0$ is small enough, then the intersection body of $K_t$ is almost the same as that of $B_\infty$ around the equator. In other words, $\psi_{I\!K_t}(0)=\psi_{I\!B_\infty}(0)+O(e^{-1/t})$ and $\psi_{I\!K_t}(\eps)=\psi_{I\!B_\infty}(\eps)+O(e^{-1/t})$ for small $\eps>0$. Nevertheless, note that $\rd{I\!B_\infty}(0)=1$, but $\rd{I\!K_\de}(0)=1/t^{n-1}$, i.e., they have quite different radial functions on the axis as $t$ approaches to zero. So, 
\begin{align*}
\de^e_{I\!K_t}(\eps) &= {\psi_{I\!K_t}(0)-\psi_{I\!K_t}(\de) \over \psi_{I\!K_t}(0)}={\psi_{I\!B_\infty}(0)-\psi_{I\!B_\infty}(\de/t^{n-1}) \over \psi_{I\!B_\infty}(0)}+O(e^{-1/t})\\
&= \de^e_{I\!B_\infty}(\eps/t^{n-1})+O(e^{-1/t}),
\end{align*}
where $\de={\rd{I\!K_t}(0)\,\eps\over\rd{I\!K_t}(\pi\!/2)} ={\rd{I\!B_\infty}(0)\,\eps\over\rd{I\!B_\infty}(\pi\!/2)}\cdot\frac1{t^{n-1}}+ O(e^{-1/t})$.
Thus, $\de^e_{I\!K_t}(\eps)/\de^e_{I\!B_\infty}(\eps)=O(t^{2-2n})$, which tends to infinity as $t$ tends to zero.  Therefore, this example shows that the constant $c_K$ in Theorem \ref{thm:uniform_convexity_IK} may be unbounded in case of star bodies.
\end{remark}

\section{Double intersection bodies of revolution in high dimension}

Recently, Fish, Nazarov, Ryabogin, and Zvavitch \cite{FNRZ} proved that the iterations of the intersection body operator, applied to any symmetric star body sufficiently close to a Euclidean ball $B_2^n$ in the Banach-Mazur distance, converge to $B_2^n$ in the Banach-Mazur distance. Namely, if $K$ is a star body in $\R^n$ with $d_{B\!M}(K,B_2^n)=1+\eps$ for small $\eps>0$, then $$\lim_{m\to\infty}d_{B\!M}(I^mK,B_2^n)=1.$$
In case of bodies of revolution in sufficiently high dimension, it turns out that it is enough to apply the intersection body operator {\em twice} to get close to the Euclidean ball in the Banach-Mazur distance, which will be shown in this section. The uniform boundedness of the constant $c_K$ by absolute constants in Theorem \ref{thm:uniform_convexity_IK} plays an important role in the following result.

\begin{theorem}\label{thm:double_intersection_body}
Let $K$ be a symmetric convex body of revolution in $\R^n$. Then the double intersection body $I^2K$ is close to an ellipsoid if the dimension $n$ is large enough. More precisely, for every $\eps>0$ there exists an integer $N>0$ such that for every $n\ge N$ and any body $K\subset\R^n$ of revolution, $$d_{B\!M}(I^2K,B_2^n)\le1+\eps.$$ %where $B_2^n=\set{x\in\R^n:|x|\le1}$ is the unit Euclidean ball in $\R^n$.
\end{theorem}

\begin{proof}
By $B$ we denote the unit ball in $\R^n$ (instead of $B_2^n$).
It follows from Theorem~\ref{thm:uniform_convexity_IK} that $\de^e_{I\!K}(\eps) = c_K\eps^2 + O(\eps^3)$ where $c_1<c_K<c_2$ for absolute constants $c_1,c_2>0$. Also note $\de^e_{B}(\eps)=\eps^2/2+O(\eps^3)$ for the unit ball $B$. Consider a linear transformation $T$ (dilation) which gives $\rd{T(I\!K)}(\pi\!/2)=1$ and $\rd{T(I\!K)}(0)=1/\sqrt{2c_K}$. 
Denote $L:=T(I\!K)$. Then, 
\begin{equation}\label{eq:psi_L}
\psi_L(t)=1-\de^e_L\brR{t/\sqrt{2c_K}}+O(t^3)=1-t^2/2 +O(t^3).
\end{equation}
Also, it is not hard to compute the function $\psi_{B}$ for the ball $B$,
\begin{equation}\label{eq:psi_Ball}
\psi_{B}(t)=\frac1{\sqrt{1+t^2}}=1-t^2/2 +O(t^3).
\end{equation}
Let 
\begin{align*}
\si_L&=\psi_L^{-1}(1-1/n)=\sqrt{2/n}+o(n^{-1/2})\\
\si_B&=\psi_{B}^{-1}(1-1/n)=\sqrt{2/n}+o(n^{-1/2}).
\end{align*}
Fix $\eps>0$, and let $R=-4\log\eps$, $N=R^2/\eps^4$. Then, we claim that for every $n\ge N$,
\begin{equation}\label{eq:claim_double_intersection}
\abs{\frac{\rd{I\!L}(\te)}{\rd{I\!B}(\te)}-1}\le\eps \quad\forall\te\in[0,\pi\!/2].
\end{equation}
Note that $\rd{I\!L}(0)=\rd{I\!B}(0)\approx\sqrt{\pi\over2n}$ for large $n$ by \eqref{eq:gardner_IK_axis}. In case that $\te$ is close to $0$, the claim \eqref{eq:claim_double_intersection} can be obtained from the convexity of $I\!L$. Thus it suffices to consider all angles $\te$ with $$\tan\te\ge\eps.$$
For $n\ge N$, since $(1+R/n)^{2-n}\le\brR{e^{\frac{R}{2n}}}^{2-n}\le e^{1-R/2}=e\eps^2$, we have 
\begin{equation}\label{eq:remainder_eps}
\brS{1+\frac{R}{n}}^{2-n}=O(\eps^2).
\end{equation}
Note also that for $n\ge N$, since $\si_L$, $\si_B$ are bounded by $\sqrt{2/N}=\sqrt{2}\eps^2/R$,
\begin{equation}\label{eq:sigma_tau_R}
\si_L R=O(\eps^2) \quad\text{and}\quad \si_B R=O(\eps^2).
\end{equation}
Lemma \ref{lem:formula_R} and \eqref{eq:remainder_eps} give
\begin{equation*}
\frac{\rd{I\!L}(\te)\sin\te}{\si_L} = \int_0^R \psi_L(\si_Lt)^{n-1}\brS{1-\frac{\si_L^2t^2}{\tan^2\te}}^{\frac{n-4}2}dt  + O(\eps^2).
\end{equation*}
Note that $\rd{I\!L}(\pi\!/2)$ is comparable to $\si_L$ by Lemma \ref{lem:approx}, and $\rd{I\!L}(0)\approx\sqrt{\pi\over2n}$ is also comparable to $\si_L$. So, by convexity of $I\!L$, the radial function for $I\!K$ at any angle is comparable to $\si_L$. Moreove, since 
\begin{align*}
\si_L\eps^2&=\frac{\si_L}{\rd{I\!L}(\te)}\cdot\frac{\eps^2}{\sin\te}\cdot\rd{I\!L}(\te)\sin\te \le \frac{\si_L}{\rd{I\!L}(\te)}\cdot2\eps\cdot\rd{I\!L}(\te)\sin\te\\
&=O(\eps)\cdot\rd{I\!L}(\te)\sin\te,
\end{align*}
we have
\begin{align*}
\rd{I\!L}(\te)\sin\te &= \int_0^{\si_LR} \psi_L(t)^{n-1}\brR{1-t^2/\tan^2\te}^{\frac{n-4}2}dt  + O(\si_L\eps^2)\\
&= (1+O(\eps))\int_0^{\si_LR} \psi_L(t)^{n-1}\brR{1-t^2/\tan^2\te}^{\frac{n-4}2}dt.
\end{align*}
Similarly, we have the same equality for $I\!B$. Without loss of generality, we may assume $\si_L\ge\si_B$. Then,
\begin{align}
\rd{I\!L}(\te)\sin\te &= (1+O(\eps))\int_0^{\si_LR} \psi_L(t)^{n-1}\brR{1-t^2/\tan^2\te}^{\frac{n-4}2}dt\\
\rd{I\!B}(\te)\sin\te &= (1+O(\eps))\int_0^{\si_LR} \psi_{B}(t)^{n-1}\brR{1-t^2/\tan^2\te}^{\frac{n-4}2}dt
\end{align}
Moreover, \eqref{eq:psi_L} and \eqref{eq:psi_Ball} give
$$\brR{\psi_L(t)\over\psi_{B}(t)}^{n-1}=\brS{1+O(\si_L^3R^3)}^{n-1}=1+O(n\si_L^3R^3).$$
Here, since $n\si_L^2\le3$, $\eps R^2=16\eps(\log\eps)^2\le1$, and $\si_LR=O(\eps^2)$, we get
$$n\si_L^3R^3=(n\si_L^2)(\eps R^2)(\si_LR/\eps)=O(\eps).$$
Finally, we have 
\begin{equation*}
{\rd{I\!L}(\te)\over\rd{I\!B}(\te)} =1+O(\eps)\quad\text{for each angle } \te,
\end{equation*}
which completes the proof.
\end{proof}

\begin{remark}
Theorem \ref{thm:double_intersection_body} says that the double intersection body of any body of revolution becomes close to an ellipsoid as the dimension increases to the infinity. However, it is not true for the single intersection body, in general. For example, consider the cylinder $B_\infty$.
Then the Banach-Mazur distance between $I\!B_\infty$ and $B_2^n$ does not converge to 1 as $n$ tends to the infinity.
\end{remark}
\begin{proof}
The function $\psi_{B_\infty}$ for the cylinder $B_\infty$ is given by $\psi_{B_\infty}(t)=\min(1,1/t)$, as in the proof of Lemma \ref{lem:easy_bounds_psi_K}. Note that $\rd{I\!B_\infty}(0)=\sqrt{\pi/(2n)}$ by \eqref{eq:gardner_IK_axis}, and
\begin{equation}
\rd{I\!B_\infty}(\pi\!/2)=\int_0^\infty \psi_{B_\infty}(t)^{n-1}=\int_0^1 1dt+\int_1^\infty t^{1-n}dt=\frac{n-1}{n-2}.
\end{equation}
Choose the angle $\te$ with $\tan\te=\frac{\rd{I\!B_\infty}(\pi\!/2)}{\rd{I\!B_\infty}(0)}$, and let $x=\cot\te$. Then 
\begin{equation}
x=\frac{\rd{I\!B_\infty}(0)}{\rd{I\!B_\infty}(\pi\!/2)}=\sqrt{\frac\pi{2n}}\cdot\frac{n-2}{n-1}=O(1/\sqrt{n}).
\end{equation}
In addition,
\begin{equation*}
\rd{I\!B_\infty}(\te)\sin\te =\psi_{I\!B_\infty}(x)=\int_0^{1/x}\psi_{B_\infty}(t)^{n-1}(1-x^2t^2)^{\frac{n-4}2}dt= {\rm (I)} + {\rm (II)},
\end{equation*}
where  $${\rm (I)}=\int_0^1(1-x^2t^2)^{\frac{n-4}2}dt\quad\text{and}\quad{\rm (II)}=\int_1^{1/x}t^{1-n}(1-x^2t^2)^{\frac{n-4}2}dt.$$
For the first term, note that
\begin{align*}
\lim_{n\rightarrow\infty}{\rm (I)} &= \lim_{n\rightarrow\infty}\int_0^1\brR{1-\frac{\pi}{2}\frac{(n-2)^2}{(n-1)^2}\cdot\frac{t^2}{n}}^{\frac{n-4}2}dt \\
&=\int_0^1 e^{-\frac{\pi}{4}t^2}dt \ge \int_0^1(1-\pi t^2/4)dt=1-\pi/12.
\end{align*}
The second term
\begin{align*}
{\rm (II)} &= \int_1^{1/x}t^{1-n}(1-x^2t^2)^{\frac{n-4}2}dt= \int_1^{1/x}\frac1{t^3}\brR{\frac1{t^2}-x^2}^{\frac{n-4}2}dt \\
&= \frac12\int_0^{1-x^2}s^{\frac{n-4}2}ds = \frac1{n-2}(1-x^2)^{\frac{n-4}2} =  \frac1{n-2}\brR{1-\frac{\pi}{2}\frac{(n-2)^2}{(n-1)^2}\cdot\frac1n}^{\frac{n-4}2}
\end{align*}
converges to zero as $n$ tends to infinity. 
Let $L$ be the body of revolution obtained by shrinking $I\!B_\infty$ by $\rd{I\!B_\infty}(0)$ on $\spn\set{e_1}$ and by $\rd{I\!B_\infty}(\pi\!/2)$ on $e_1^\perp$. That is, $\rd{L}(0)=\rd{L}(\pi\!/2)=1$. Then
\begin{align*}
\rd{L}(\pi/4)\sin(\pi/4) &=\frac{\rd{I\!B_\infty}(\te)\sin\te}{\rd{I\!B_\infty}(\pi\!/2)},
\end{align*}
and it limit as $n\rightarrow\infty$ is given by $$\lim_{n\rightarrow\infty}\frac{\rm (I) + (II)}{(n-1)/(n-2)}\ge 1-\pi/12.$$ Thus, we get $\rd{L}(\pi/4)\ge\sqrt{2}(1-\pi/12)>1$ for large $n$. Note that  $L$ is a symmetric body  of revolution about the axis $e_1$ satisfying $\rd{L}(0)=\rd{L}(\pi\!/2)=1$ and $\rd{L}(\pi/4)=c>1$, which implies that $d_{B\!M}(L,B_2^n)\ge c>1$ for large $n$. Therefore $$\lim_{n\rightarrow\infty}d_{B\!M}(I\!B_\infty,B_2^n)=\lim_{n\rightarrow\infty}d_{B\!M}(L,B_2^n)\ge c>1.$$
\end{proof}

\hide{
\section{Iterations of Intersection Body Operator}
\begin{theorem}[FNRZ]
For every integer $n\ge3$, there exists $\eps>0$ such that for any star body $K$ in $\R^n$ with $d_{BM}(K,B_2^n)\le\eps$ $$d_{BM}(I^mK,B_2^n)\rightarrow1 \quad\text{as}\quad m\rightarrow\infty.$$
\end{theorem}
In order to induce the convergence of $I^mK$ for bodies of revolution $K$, we may ask\\
\noindent{\bf Question} In a class of revolution bodies of any dimensions, is it possible to take a threshold $\eps$ of convergence in [FNRZ] which is INDEPENDENT of dimension $n$?\\
If this is true, then we can conclude that the iteration of the intersection body operators on bodies of revolution converges to an ellipsoid in high dimension.
}

\section{Local convexity of intersection bodies of star bodies of revolution.}

 Let $K$ be a symmetric star body of revolution in  $\mathbb{R}^n$. Following the definition in \cite[Section 0.7]{Ga2}, the radial function $\rho_K$  is continuous, but it may attain the value zero, and hence the origin need not be an interior point of $K$. In this section we will study the local convexity at the equator of $\rho_{I\!K}$. The main result is an analogue of  Theorem \ref{thm:uniform_convexity_IK}: In dimensions five and higher, the intersection body of a star body is locally convex with equatorial power type 2. As observed in Remark 2, in the convex case the constant $c_K$ is uniformly bounded, but if $K$ is not convex $c_K$ may be arbitrarily big or close to zero by choosing $K$ appropriately. As for the four-dimensional case, $I\!K$ is still locally convex at the equator, but it may not be strictly convex (Example \ref{cyl} below), or if it is, its modulus of convexity may not be of power type 2 (Example \ref{mod4} below). However, if the origin is an interior point of $K$, then $I\!K$ has equatorial power type 2 (Theorem \ref{int0} below).

If $\rho_K$ is not identically zero, we have from \eqref{eq:gardner_IK} that $\rd{I\!K}(\pi\!/2)$ is a  positive number. Applying a dilation, we will assume that $\rd{I\!K}(\pi\!/2) =1$. Thus, to study the equatorial power type of $I\!K$, we can use the function $\psi_{I\!K}$ as in Section 2.

%We say that a star body $K$ is {\em convex} at a point $p$ on the boundary of $K$ if there exists a closed neighborhood $B(p,\eps)=\set{q\in\R^n:|p-q|\le\eps}$ of $p$ such that $K\cap B(p,\eps)$ is a convex body. Furthermore, if $p$ is an extreme point of $K\cap B(p,\eps)$, then $K$ is said to be {\em strictly convex} at the point $p$. In case that $K$ is a body of revolution, we say that $K$ is {\em locally convex at the equator} if $K$ is convex at each point on the equator. 

We start with characterization of local convexity at the equator in terms of the radial function.

%We first need two Lemmas. The first one expresses the local convexity of $K$ at the equator in terms of its radial funcion. The second one shows that if $n \geq 5$,  $\rho_{I\!K}$ has continuous second derivative at every point except at the axis of revolution.  

\begin{lemma}
\label{convex-radial-func}
Let $K$ be a symmetric star body of revolution such that $\rd{K}\in C^2[0,\pi]$. Then $K$ is locally convex at the equator if and only if $\rho_K(\pi\!/2)-\rho_K''(\pi\!/2) \geq 0$.
\end{lemma}

\begin{proof}
We express the boundary of $K$ parametrically by $\langle x(\theta),y(\theta) \rangle$, where $x(\theta)=\rho_K(\theta)\cos(\theta)$ and $y(\theta)=\rho(\theta)\sin(\theta)$. With this representation,the equator corresponds to the point $(0,\rho_K(\pi\!/2))$. Since the boundary of $K$ is of the class $C^2$, we can study the local convexity at the equator by means of the second derivative 
\[
      \frac{d^2y}{dx^2}=\frac{\frac{d}{d\theta} \left( \frac{dy}{d\theta}/ \frac{dx}{d\theta}\right)}{\frac{dx}{d\theta}}=
    - \frac{\rho_K(\theta)^2+2(\rho_K'(\theta))^2-\rho_K(\theta)\rho''(\theta)}{\left(\sin(\theta) \rho_K(\theta)- \cos(\theta) \rho_K'(\theta) \right)^3}.
\]
At the equator, $\theta=\pi\!/2$. Also, it follows from the central and axial symmetries that  $\rho_K'(\pi\!/2)=0$.  Therefore, the above expression  simplifies to 
\[
        \frac{d^2y}{dx^2}=-\frac{ \left( \rho_K(\pi\!/2)-\rho_K''(\pi\!/2) \right)}{(\rho_K(\pi\!/2))^2}
\]
Hence $K$ is locally convex at the equator if and only if $ \rho_K(\pi\!/2)-\rho_K''(\pi\!/2) \geq 0$.  
\end{proof}

\begin{lemma}
\label{regularity}
Let $K$ be a symmetric star body of revolution in $\mathbb{R}^n$, $n\geq 5$, with radial function $\rho_K$. Let $I\!K$ be the intersection body of $K$. Then $\rho_{I\!K}\in C^2(0,\pi)$. %that is, $\rho_{I\!K}(\te)$ has continuous second derivative for every $\te \in(0,\pi)$.  Here, $\rho_{I\!K}$ is defined for values of $\te$ greater than $\pi/2$  by its even extension around $\pi/2$, {\it i.e.}  $\rho_{I\!K}(\te)=\rho_{I\!K}(\pi-\te)$. 
\end{lemma}

\begin{proof}
The part of this lemma corresponding to even values of $n$ was proven in \cite{A}, Proposition 8, where it was shown that $\rho_{I\!K}$ has continuous $(n-2)/2$-th derivative at every $\te \in (0,\pi/2)$, and a continuous $(n-2)$-th derivative at the point $\te=\pi/2$. We will thus assume that $n$ is odd, $n \geq 5$.

At the point $\te=\pi/2$, we will use definition  (\ref{psi}) and prove that $\psi_{I\!K}$ has a continuous second derivative at $x=0$. If  $n \geq 7$, differentiating with respect to $x$ twice in equation (\ref{eq:formula1}) gives
\[
    \psi_{I\!K}''(x)= -(n-4) \int_0^{1/x} \psi_K(t)^{n-1}\, t^2 \brS{1-x^2t^2}^{\frac{n-6}2}dt
\]
\[
 +(n-4)(n-6)x^2 \int_0^{1/x} \psi_K(t)^{n-1}\, t^4 \brS{1-x^2t^2}^{\frac{n-8}2}dt.
\]
Since $n \geq 7$ and $\psi_K$ is a bounded function that satisfies $\psi_K(t)=O(1/t)$ as $t$ tends to infinity, the integrals are convergent at infinity and
\[
    \psi_{I\!K}''(0)= -(n-4) \int_0^\infty \psi_K(t)^{n-1}\, t^2 \,dt +(n-4)(n-6)x^2 \int_0^{\infty} \psi_K(t)^{n-1}\, t^4 \,dt <+\infty.
\]
Thus, $\psi_{I\!K}''(0)$ is finite, and since $\psi_{I\!K}(x)$ is extended evenly for negative values of $x$, $\psi_{I\!K}''$ is continuous at 0. As for $n=5$, the first derivative of $\psi_{I\!K}$ is 
\[
    \psi_{I\!K}'(x)= -x \int_0^{1/x} \psi_K(t)^4\, t^2 \brS{1-x^2t^2}^{-1/2}dt,
\]
and 
\[
     \frac{\psi_{I\!K}'(x)-\psi_{I\!K}'(0)}{x}= - \int_0^{1/x} \psi_K(t)^4\, t^2 \brS{1-x^2t^2}^{-1/2}dt,
\]
which, when $x$ approaches zero, tends to the convergent integral $\displaystyle -\int_0^\infty \psi_K(t)^4\, t^2 dt$. We have thus shown that $\psi_{I\!K}(x)$ (and $\rho_{I\!K}$) have continuous second derivative at zero for every $n \geq 5$.\\

Finally, we will show that $\rho_{I\!K}$ has continuous second derivative at every $\te \in (0,\pi/2)$.  Setting $x=\sin \te$, $t=\cos \phi$, $r(t)=\rho_K^{n-1}(\arccos(t))$ and $F(x)=\rho_{I\!K}(\arcsin \te)$ in  equation (\ref{eq:gardner_IK}), we obtain the expression
\begin{equation}
 \label{ibsbxt}
     F(x)=\frac{1}{x^{n-3}} \int_0^x r(t) (x^2-t^2)^{(n-4)/2} \, dt,
\end{equation}
for $x\in (0,1]$, and $F(0)=c_n\, r(0)$, where $c_n$ is a constant depending only on $n$. Consider a point $a\in (0,1)$ such that $r(x)$ is continuous but not differentiable at $a$. Differentiating equation (\ref{ibsbxt}) $(n-3)/2$ times, we obtain
\[
      F^{((n-3)/2)}(x)=\sum_{k=-1, k \, odd}^{n-4} p_k(x) \int_0^x r(t) \left( \sqrt{x^2-t^2} \right)^k \, dt,
\]
where each $p_k(x)$ is a rational function of the form $c(n,k)/x^{b(n,k)}$, for some constants $c,b$. Note that the denominator of $p_k(x)$ is nonzero if $x\in(0,1)$. Therefore, every term with positive $k$ is a continuous function at $a$. For the term corresponding to $k=-1$, we change variables by setting  $t=x\sin u$, so that $\displaystyle \int_0^x \frac{r(t) }{\sqrt{x^2-t^2}} \, dt$ becomes 
$\int_0^{\pi/2} r(x\sin u) \, du=:G(x)$. Then,
\begin{equation}
 \label{left}
   \lim_{x \rightarrow a-}  \frac{G(x)-G(a)}{x-a}= \int_0^{\pi/2}  \lim_{x \rightarrow a-} \left( \frac{r(x \sin u)-r(a \sin u)}{x-a} \right) \,du
\end{equation}
and if $x>a$, $\lim_{x \rightarrow a+}  \frac{G(x)-G(a)}{x-a}$
\[
=\lim_{x \rightarrow a+} \left( \frac{1}{x-a}\int_0^{\arcsin \left( \frac{a}{x} \right)} \left( r(x\sin u)-r(a \sin u) \right) \, du+ \frac{1}{x-a} \int_{\arcsin \left( \frac{a}{x} \right)}^{\pi/2}  \left( r(x\sin u)-r(a \sin u) \right) \, du \right).
\]
Note that the first tends to the right hand side of (\ref{left}), while the second term tends to zero because $r$ is continuous at $x=a$. However, for the second derivative, the corresponding term will tend to $r'_+(a)-r'_-(a)$, which is not zero, and thus $G$ does not have a continuous second derivative at $a$. We have, in fact, shown that $\rho_{I\!K}$ has $(n-1)/2$ continuous derivatives at any point $\te \in (0,\pi/2)$.
\end{proof}

\begin{remark}
 We wish to note that the local convexity and regularity properties of $I\!K$ at the axis of revolution are the same as those of $K$ at the equator. It is easily seen by calculating the intersection body of a double cone that, in general,  $\rho_{I\!K}(\te)$ is not differentiable at $\te=0$. The general argument is as follows. Setting  $t=x\sin u$  in (\ref{ibsbxt})  gives 
\[
         F(x)= \int_0^{\pi\!/2} r(x\sin u) ( \cos u )^{n-3} \,du
\]
for $x \in(0,1]$. At $x=0$,  $F(0)=r(0)  \left( \int_0^{\pi\!/2} ( \cos u )^{n-3} \,du \right)$.  
Then,
\begin{equation}
 \label{reg-0-right}
   \lim_{x \rightarrow 0+} \frac{F(x)-F(0)}{x}= \int_0^{\pi\!/2}  \left(  \lim_{x \rightarrow 0+}  \frac{ r(x\sin u)-r(0)}{x} \right) ( \cos u )^{n-3} \,du,
\end{equation}
and similarly for the left-hand side limit. If the function $\rho_K$  (and hence $r$) is differentiable at zero, then so is $F$. However, if the right and left hand side limits of $r$ take different values, then the same will be the case for $F$. Observe also that (\ref{reg-0-right}) implies that the local convexity of $r$ at the equator and $F$ at the axis present the same behavior. In particular, if the body $K$ is not locally convex at the equator, then $I\!K$ will not be locally convex at the axis of revolution. 
\end{remark}

Now we are ready to prove the result on local convexity of $I\!K$ at the equator.

\begin{theorem}
 \label{Starbodies}
   Let $K$ be a symmetric star body of revolution in $\R^n$, $n \geq 5$. Then its intersection body  $I\!K$ is strictly convex at the equator, with equatorial power type 2.
\end{theorem}

\begin{proof}

By Lemma \ref{regularity}, $\rho_{I\!K}(\te)$ has continuous second derivative for every $\te \in (0,\pi/2]$. Observe that $\left(1-\frac{\cos^2 \phi}{\sin^2 \te} \right)<(1-\cos^2 \phi)$ for every $\te<\pi\!/2$. Using this estimate in equation (\ref{eq:gardner_IK}), we obtain
\[
 \rd{I\!K}(\te)\sin\te = \int_{\pi\!/2-\te}^{\pi\!/2} \rd{K}(\phi)^{n-1}\brS{1-\frac{\cos^2\phi}{\sin^2\te}}^{\frac{n-4}{2}}\sin\phi \, d\phi
\]
\[
<  \int_{\pi\!/2-\te}^{\pi\!/2} \rd{K}(\phi)^{n-1}(1-\cos^2 \phi)^{\frac{n-4}{2}}\sin\phi  \, d\phi 
\]
\[
 \leq \int_{0}^{\pi\!/2} \rd{K}(\phi)^{n-1}(1-\cos^2 \phi)^{\frac{n-4}{2}}\sin\phi  \, d\phi  
  = \rd{I\!K}(\pi\!/2)=1.
\]
Therefore, $1-\psi_{I\!K}(\eps)=1-\ \rd{I\!K}(\te)\sin(\te)$ has a local minimum at $\eps=0$, and thus $-\psi_{I\!K}''(0)=\rho_{I\!K}(\pi\!/2)-\rho_{I\!K}''(\pi\!/2) \geq 0$. By Lemma \ref{convex-radial-func}, $I\!K$ is locally convex at the equator.\\

Assume that  $\rho_{I\!K}(\pi\!/2)-\rho_{I\!K}''(\pi\!/2)=0$. We claim that this contradicts the fact that $I\!K$ is an intersection body by using Koldobsky's Second Derivative test \cite{K2} (see also \cite[Theorem 4.19]{Ko}).  Indeed, since $I\!K$ is a body of revolution, if we consider the coordinates $(x_1,\overline{x})$ in $\mathbb{R}^n$, where $\overline{x}=(x_2,\ldots,x_{n})$ and  $x_1$ is in the direction of the axis of revolution, then the Minkowsky functional of $I\!K$ is given by 
\[
\|(x_1,\overline{x})\|_{I\!K}^{-1}=\frac{1}{\sqrt{x_1^2+\overline{x}^2}} \, \rho_{I\!K} \left(\arccos \left( \frac{x_1}{\sqrt{x_1^2+\overline{x}^2}} \right) \right),  
\]
and the condition of the Second Derivative test, 
\[
\frac{\partial^2 (\|(x_1,\overline{x})\|_{I\!K})}{\partial x_1^2}(0,x_2,\ldots,x_n)=0,
\]
is easily computed to be equivalent to $\rho_{I\!K}(\pi\!/2)-\rho_{I\!K}''(\pi\!/2)=0$. Besides, the convergence of  $\frac{\partial^2 (\|x\|_{I\!K})}{\partial x_1^2}$ to  0 as $x_1$ approaches zero  is uniform in a neighborhood of the equator by Lemma \ref{regularity}. Hence, the Second Derivative test implies that $I\!K$ is not an intersection body, obviously a contradiction. Therefore, $\rho_{I\!K}(\pi\!/2)-\rho_{I\!K}''(\pi\!/2)> 0$, and  $I\!K$ is strictly convex at the equator, with equatorial power type 2.

\end{proof}

In four dimensions, the result of Theorem \ref{Starbodies} is not necessarily true, as the following two examples show. 

\begin{example} 
 \label{cyl}
When the origin is not an interior point of $K$, $I\!K$ need not be strictly convex at the equator. 
Assume  that  $\rho_K(\phi)=0$ for all $\phi \in [0,\al]$, $\al>0$. Then, if $\te \in [\pi\!/2-\al,\pi\!/2]$, 
\[
\rd{I\!K}(\te)\sin\te = \int_{\pi\!/2-\te}^{\pi\!/2} \rd{K}(\phi)^3 \sin\phi \, d\phi=\int_{0}^{\pi\!/2} \rd{K}(\phi)^{3}\sin\phi \, d\phi  =C.
\]
Therefore, $\rd{I\!K}(\te)=C/\sin \te$ for all $\te  \in [\pi\!/2-\al,\pi\!/2]$, which means that $I\!K$ is cylindrical around the equator. Hence, $I\!K$ is locally convex at the equator, but not strictly convex.
\end{example}

\begin{figure}[h!]
\begin{center}
\includegraphics[scale=1]{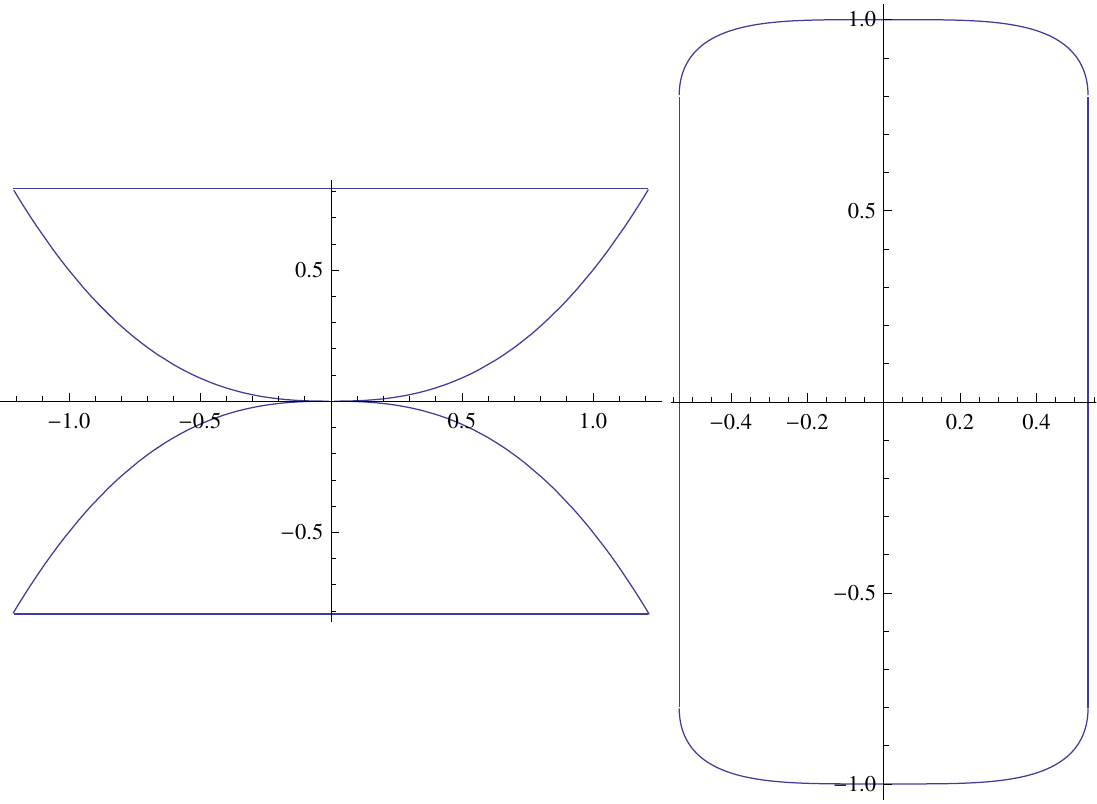}
\end{center}
\caption{\small The bodies $K$ (left) and $I\!K$ (right) in Example \ref{mod4}. }
\label{modconv4}
\end{figure}

\begin{example} 
 \label{mod4}
In this example we present a four-dimensional intersection body of a star body, which is strictly convex but does not have modulus of convexity of power type 2. Figure \ref{modconv4} shows the body of revolution $K$ and its four-dimensional intersection body  $I\!K$. The radial function if $K$ is 
\[
      \rho_{K}(\theta)=\left\{ \begin{array}{cc}
               (4 \sin^2{u}/\cos^5{u})^{1/3},  & \mbox{ if } 0 \leq \te \leq \pi/2 - \arctan(\sqrt[4]{5})\\
               A/\sin{u}, &  \mbox{ if }  \pi/2 - \arctan(\sqrt[4]{5})\leq \te \leq \pi/2.
     \end{array}
 \right.
\]

The radial function of $I\!K$, which we calculated using  {\it Mathematica}, is 
\[
      \rho_{I\!K}(\theta)=\left\{ \begin{array}{cc}
                B/\cos\te, & \mbox{ if } 0 \leq \te \leq \arctan(\sqrt[4]{5})\\
                (2(\sin \te)^2-1)/(\sin\te)^5, & \mbox{ if } \arctan(\sqrt[4]{5}) \leq \te \leq \pi\!/2,
     \end{array}
 \right.
\]
where $A,B$ are constants chosen appropriately so that $\rho_K,\rho_{I\!K}$ are continuous.  If we compute the modulus of convexity at the equator for $I\!K$, we obtain
\[
    \psi_{I\!K}(\eps)=1-\rho_{I\!K}(\te) \sin\te=1- \frac{2(\sin \te)^2-1}{(\sin\te)^4}=(\cot \te)^4=\eps^4,
\]
where the last step comes from the definition of $\psi_{I\!K}$ (\ref{psi}). Hence $I\!K$ is strictly convex at the equator, but has equatorial power type 4.
\end{example}

The bodies in Examples 1 and 2 have the common feature that the origin is not an interior point of $K$. With the additional hypothesis that the origin is an interior point of $K$, the intersection body of $K$ has equatorial power type 2, even in the four dimensional case, as shown in the next Theorem. Note that none of the Theorems \ref{Starbodies} and \ref{int0} implies the other one: Theorem \ref{Starbodies} assumes dimension $n \geq 5$, but allows the origin to be a boundary point of $K$, while the result of Theorem \ref{int0} applies for $n \geq 4$, while needing that the origin is interior to $K$. The proof of  Theorem \ref{int0}  uses similar ideas as those in the proof of Theorem \ref{thm:uniform_convexity_IK}.

\begin{theorem}
 \label{int0}
 Let $K$ be a symmetric star body of revolution in $\mathbb{R}^n$, for $n\geq 4$, such that the origin is an interior point of $K$. Then $I\!K$ has equatorial power type 2.
\end{theorem}

\begin{proof}
Since the origin is an interior point, we may assume that $\rd{K}(0)=1$ and $rB_\infty\subset K\subset RB_\infty$ for some constants $r,R>0$ depending on $K$ where $B_\infty=\set{(x,y)\in\R^n=\R\times\R^{n-1}:|x|\le1, |y|\le1}$. For $\te\in[0,\pi\!/2]$, consider the symmetric convex body $K_\te$ defined by
\begin{equation*}
\rd{K_\te}(\varphi)=
\begin{cases}
\rd{K}(\varphi),&\quad 0\le\varphi\le\te \\
\rd{L_\te}(\varphi),&\quad \te\le\varphi\le\pi\!/2
\end{cases}
\end{equation*}
where $L_\te$ is a body of revolution obtained by rotating the line containing two points of angles $0$, $\te$ on the boundary of $K$, i.e.,  $$L_\te=\set{(x,y)\in\R^n=\R\times\R^{n-1}:|x|+b|y|=1}\quad\text{for }b=b(\te)=\frac{1-\rd{K}(\te)\cos\te}{\rd{K}(\te)\sin\te}.$$ Then, by \eqref{eq:psi_for_line} in Lemma \ref{lem:easy_bounds_psi_K}, we have 
\begin{equation}\label{eq:psi.theta1}
\psi_{K_\te}(x)=\psi_{L_\te}(x)=\frac1{x+b},\quad\text{for every }x\ge\cot\te,
\end{equation}
and moreover, from $rB_\infty\subset K\subset RB_\infty$,
\begin{equation}\label{eq:psi.theta2}
\frac{r}{x}\le \psi_{K_\te}(x)\le\frac{R}{x},\quad\text{for every }x\ge1.
\end{equation}
We need to compute $\psi_{I\!K}(0)-\psi_{I\!K}(\de)$ for small $\de$:
\begin{align*}
\psi_{I\!K_\te}(0)-\psi_{I\!K_\te}(\de) &=\int_{1/\de}^\infty \psi_{K_\te}(t)^{n-1}dt + \int_0^{1/\de} \psi_{K_\te}(t)^{n-1}\brS{1-(1-\de^2t^2)^{\frac{n-4}2}}dt \\
&= \text{\quad(I)\qquad\quad+\qquad\quad(II)}. 
\end{align*}
If $\te\le\pi/2$ and $\de<\tan\te$, then \eqref{eq:psi.theta2} gives upper/lower bounds of the firs term: $$\int_{1/\de}^\infty(r/t)^{n-1}dt \le\text{(I)}\le\int_{1/\de}^\infty(R/t)^{n-1}dt.$$ So, the first term is bounded by $(r^{n-1}/(n-2))\de^{n-2}$ and $(R^{n-1}/(n-2))\de^{n-2}$, which are independent of $\te$. If $n=4$, then (I) is asymptotically equivalent to $\de^2$ and the second term (II) is equal to zero. Assume $n\ge5$. Then the second term (II) is divided into two parts as follows.
\begin{align*}
\text{(II)} &= \int_0^{\cot\te} + \int_{\cot\te}^{1/\de} \psi_{K_\te}(t)^{n-1}\brS{1-(1-\de^2t^2)^{\frac{n-4}2}}dt \\
&= \text{(II-1)\,\,\,\,+\,\,\,(II-2)},
\end{align*}
where 
\begin{align*}
\text{(II-1)} &= \int_0^{\cot\te} \psi_{K_\te}(t)^{n-1}\brS{1-(1-\frac{n-4}2\de^2t^2+O(\de^4))}dt \\
&=\brR{\frac{n-4}{2}\int_0^{\cot\te} \psi_{K_\te}(t)^{n-1}t^2dt}\de^2,
\end{align*}
and
\begin{align*}
\text{(II-2)} &\asymp \int_{\cot\te}^{1/\de} \frac{1-(1-\de^2t^2)^{\frac{n-4}2}}{t^{n-1}}dt =\int_{\cot\te}^{1/\de} \frac1{t^{n-1}}dt - \int_{\cot\te}^{1/\de} \frac{(1-\de^2t^2)^{\frac{n-4}2}}{t^{n-4}}\,\frac{dt}{t^3}\\
&=\frac{(\tan\te)^{n-2}-\de^{n-2}}{n-2} - \frac12\int^{\tan^2\te}_{\de^2} (s-\de^2)^{\frac{n-4}2}ds\\
&=\frac{(\tan\te)^{n-2}-\de^{n-2}-(\tan^2\te-\de^2)^{\frac{n-2}2}}{n-2} =\frac{(\tan\te)^{n-4}}{2}\de^2+O(\de^{n-2}).
\end{align*}
(Note that $\psi_{K_\te}(t)$ in the second integral (II-2) is comparable to $1/t$ by \eqref{eq:psi.theta2}).
Furthermore, when $n\ge5$, the integral of (II-1) is bounded above and below by positive constants independent of $\te$:
$$\int_0^{\cot\te} \psi_{K_\te}(t)^{n-1}t^2dt \le \int_0^1R^{n-1}t^2dt + \int_1^\infty(R/t)^{n-1}t^2dt=\frac{n-1}{3(n-4)}R^{n-1}$$ and $$\int_0^{\cot\te} \psi_{K_\te}(t)^{n-1}t^2dt \ge \int_0^1r^{n-1}t^2dt=\frac13 r^{n-1}.$$
Finally, letting $\te$ go to zero, we have equatorial power type 2 for the body $K$.
\end{proof}

\small{
\noindent M.A. Alfonseca: Department of Mathematics, North Dakota State University, Fargo, ND 58018, \newline Email: maria.alfonseca@ndsu.edu
\newline
\noindent J. Kim: Department of Mathematics, Kent State
University, Kent, OH 44242, USA.\newline Email:  jkim@math.kent.edu
}

\end{document}